\documentclass[11pt]{article}
\usepackage[margin=1in]{geometry}
\usepackage{amsmath,amssymb,amsthm}
\usepackage{xurl}
\usepackage[hidelinks,breaklinks=true]{hyperref}
\numberwithin{equation}{section}
\newcommand{\leanus}{\textunderscore\allowbreak}
\newcommand{\leandecl}[2]{\href{https://github.com/ritsumei-aoi/osp-triviality/blob/v2-lean-formalization/InhomogeneousDeformations/#1.lean}{\texttt{\let\_\leanus #2}}}
\newtheorem{theorem}{Theorem}[section]
\newtheorem{proposition}[theorem]{Proposition}
\newtheorem{lemma}[theorem]{Lemma}
\newtheorem{corollary}[theorem]{Corollary}
\theoremstyle{remark}
\newtheorem{remark}[theorem]{Remark}
\title{On the triviality of inhomogeneous deformations of $\mathfrak{osp}(1|2n)$}
\author{Hisashi Aoi\thanks{Department of Mathematical Sciences, Ritsumeikan University}}
\date{}
\begin{document}
\maketitle

\begin{abstract}
We specify a symmetrized mixed-oscillator deformation family of
$B(0,n)=\operatorname{osp}(1|2n)$, with even mixed coefficients and one
odd square-zero parameter. For every $n\geq1$, we derive its bracket from a
faithful oscillator realization and exhibit an odd cochain whose coboundary
is the recovered deformation coefficient. The resulting even change of generators
is an exact isomorphism over the exterior parameter algebra. For $n=1$, the
cochain agrees with the normalization of Bakalov--Sullivan. We give the source
relations and the even-central specialization explicitly, together with a
Lean~4 formalization.
\end{abstract}

\section{Introduction}\label{sec:introduction}

Bakalov and Sullivan introduce, in \cite{BS17}, a theory of supersymmetric
bilinear forms $(\cdot|\cdot)$ on a superspace $V=V_{\bar0}\oplus V_{\bar1}$
that they call \emph{inhomogeneous} when $(V_{\bar0}|V_{\bar1})\neq0$ and
$(V_i|V_i)\neq0$ for at least one $i\in\{\bar0,\bar1\}$, that is, when the
form is neither even nor odd. From a skew-supersymmetric \emph{pre-oscillator
form} --- an inhomogeneous form whose restrictions to $V_{\bar0}$ and
$V_{\bar1}$ are nondegenerate --- they further construct a central
extension $V\oplus\mathbb CK\oplus\mathbb C\kappa$ with an even central
element $K$ and an odd central element $\kappa$, giving a bracket valued
in $K$ when the parities of a generator pair agree and in $\kappa$ when
they differ (\cite[Section~4]{BS17}). Among the relations of
\eqref{eq:source}, $[b_u,b_v]=J_{uv}1$ corresponds to the homogeneous
Weyl-type commutation relation that also arises from an ordinary (even)
bilinear form with no mixed block, while $[b_u,a]=\beta_u\kappa$ is
precisely the odd central-valued term added by the mixed block
$(b_u|a)\neq0$; it embodies inhomogeneity in this sense.

Bakalov--Sullivan introduce mixed oscillator brackets and give a trivial
abelian extension of $\operatorname{osp}(1|2)$ in \cite[Section~4]{BS17}.
We generalize their deformation, and we prove that the corresponding
extension for the family with $n$ oscillator pairs is likewise trivial.

That the extension splits is not unexpected. The simple finite-dimensional Lie
superalgebras over $\mathbb C$ admitting deformations with odd parameters are
classified in \cite[Theorem~7.1]{Lei23}, the only ones being
$\mathfrak{svect}(0|2n+1)$; and an abelian extension of a Lie superalgebra
$\mathfrak s$ by its adjoint module with parities reversed is classified by the
odd part of $H^2(\mathfrak s;\mathfrak s)$. Over $\mathbb C$ --- the setting
Corollary~\ref{cor:complex} specializes to --- a splitting therefore exists for
$\mathfrak{osp}(1|2n)$ already.

The present theorem is a construction rather than a vanishing statement, and it
differs from that in three respects. The vanishing of a class exhibits no
primitive: we give $f_\beta$ explicitly, and with it an even change of
generators that is an exact isomorphism. The statement is over the universal
ring $R=P\oplus\kappa P$ with the $\beta_u$ indeterminate. And
$\Gamma_\beta$ is recovered from a faithful oscillator realization rather than
postulated, so the theorem concerns the coefficient the source actually
produces --- which no statement about cohomology classes addresses.

Under a general inhomogeneous bilinear form, there is some freedom in how the
parameters $\beta_u$ enter and how the generators are symmetrized. We
take the symmetric products $\widehat L_{uv}$, $\widehat F_u$ of
\eqref{eq:lifts} in the Weyl-algebra generators $b_u$ and the odd generator
$a$, fixing how $\beta_u$ and $\kappa$ enter so that they satisfy the
relations of \eqref{eq:source}. This symmetrization is chosen so that, for
$n=1$, it agrees with the normalization of \cite[Section~4]{BS17} (see
Section~\ref{sec:comparison}), guaranteeing that the extension to general
rank $n$ recovers it as a specialization.

We fix the fermionic square, take even coefficients in the mixed commutators,
and use symmetrized oscillator lifts. The odd exterior parameter records the
deformation, while the even central generator is specialized to the unit.
These choices specify the family under consideration and the coefficient
algebra throughout the argument.

The argument is self-contained: we include the ordered basis of the Weyl
algebra, the source isomorphism, independence of the lifts, bracket recovery
and the coboundary calculation. It is uniform in the number of oscillator
pairs.

\section{Sources and normalized generators}\label{sec:sources}

Fix $n\geq1$. Put
\[
 P=\mathbb Q[\beta_1,\ldots,\beta_{2n}],\qquad
 R=P\oplus\kappa P=P\otimes_{\mathbb Q}\Lambda_{\mathbb Q}(\kappa),
 \qquad \kappa^2=0.
\]
All $\beta_u$ are even and $\kappa$ is odd. Thus
$R$ is a supercommutative algebra. Coefficients are written on the left.
For homogeneous elements, writing $p(\cdot)\in\{0,1\}$ for the parity, we use
$[X,Y]=XY-(-1)^{p(X)p(Y)}YX$; braces denote this bracket when both arguments
are odd. Let
\[
 J=\operatorname{diag}(J_2,\ldots,J_2),\qquad
 J_2=\begin{pmatrix}0&1\\-1&0\end{pmatrix}.
\]
Indices run from $1$ to $2n$ unless otherwise stated.

Let $A_\beta$ be the unital associative $R$-superalgebra generated by even
$b_u$ and odd $a$, subject to
\begin{equation}\label{eq:source}
 [b_u,b_v]=J_{uv}1,\qquad [b_u,a]=\beta_u\kappa,
 \qquad a^2=\tfrac{1}{2} 1.
\end{equation}
The scalar algebra is supercentral: in particular
$a\kappa=-\kappa a$ and $b_u\kappa=\kappa b_u$.
Define its symmetrized elements
\begin{equation}\label{eq:lifts}
 \widehat L_{uv}=\tfrac{1}{4}(b_ub_v+b_vb_u)=\widehat L_{vu},\qquad
 \widehat F_u=\tfrac{1}{4}(ab_u+b_ua).
\end{equation}
The hats distinguish actual source elements from the coordinate symbols used below.

\begin{remark}[Even-central specialization]\label{rem:central}
Our constants in \eqref{eq:source} impose $K=1$ in an oscillator presentation
with even central generator $K$. This is specialization by the ideal generated
by $K-1$, not omission of scalars. For example, before specialization, the
single-pair definitions $H=(b_2b_1+b_1b_2)/4$ and $E^+=b_2^2/2$ give
$[H,E^+]=KE^+$, not $E^+$. All normalized tables in this paper are stated
after specialization, consistent with the unit-central Fock normalization in
\cite[Section~4]{BS17}. The independent odd parameter $\kappa$, and terms such
as $\kappa H$, are retained.
\end{remark}

Let $W_n$ be the rational Weyl algebra on even $B_u$, with
$[B_u,B_v]=J_{uv}1$. Let
\[
 C=\mathbb Q[a]/(a^2-\tfrac{1}{2}),\qquad
 A_0=W_n\otimes_{\mathbb Q}C,\qquad A_B=R\otimes_{\mathbb Q}A_0.
\]
Here $a$ is the same odd generator introduced above, and tensor products use
graded signs. The algebra $C$ has basis $1,a$, with $a^2=\tfrac{1}{2}$. Since
$W_n$ is purely even, $aB_u=B_ua$ in $A_0$ and $A_B$.

\begin{lemma}[Untwisting the source]\label{lem:source-isomorphism}
The following are mutually inverse even $R$-algebra maps, fixing $a$ and
$R$:
\[
 \Phi:A_\beta\longrightarrow A_B,\quad b_u\longmapsto B_u+\beta_u\kappa a,
 \qquad
 \Psi:A_B\longrightarrow A_\beta,\quad B_u\longmapsto b_u-\beta_u\kappa a.
\]
\end{lemma}
\begin{proof}
Set $c=\kappa a$. In both presentations, $c$ is even,
$c^2=0$ and $[c,a]=\kappa$. In $A_B$, $[B_u,c]=0$; in
$A_\beta$, $[b_u,c]=\kappa[b_u,a]=\beta_u\kappa^2=0$.
The shifted $B_u$'s therefore retain their Weyl commutators and have mixed
commutators $\beta_u\kappa$; the shifted $b_u$'s retain their Weyl
commutators and commute with $a$. The square and scalar-supercommutation
relations are also preserved. Thus both assignments induce homomorphisms.
They fix $c$, so their shifts cancel on every generator in both compositions.
\end{proof}

Write
\[
 L^0_{uv}=\tfrac{1}{4}(B_uB_v+B_vB_u),\qquad F^0_u=\tfrac{1}{2}aB_u
\]
in $A_0$. Let $\mathfrak g$ be the rational superspace with even basis
$L_{uv}$, $u\leq v$, and odd basis $F_u$, extending the notation by
$L_{vu}=L_{uv}$.

\begin{lemma}[The undeformed algebra]\label{lem:base}
The map $\iota_0:\mathfrak g\to A_0$ defined by
$L_{uv}\mapsto L^0_{uv}$, $F_u\mapsto F^0_u$ is injective, with image closed
under the supercommutator. The induced bracket is
\begin{align}
 \{F_u,F_v\}_0&=\tfrac{1}{2} L_{uv},\label{eq:base-ff}\\
 [L_{uv},F_w]_0&=\tfrac{1}{2}(J_{vw}F_u+J_{uw}F_v),\label{eq:base-lf}\\
 [L_{uv},L_{wz}]_0&=\tfrac{1}{2}\bigl(
 J_{vw}L_{uz}+J_{uw}L_{vz}
 +J_{vz}L_{uw}+J_{uz}L_{vw}\bigr).\label{eq:base-ll}
\end{align}
\end{lemma}
\begin{proof}
For one pair, with $[y,x]=1$, the monomials $x^iy^j$ span, because any
word can be reordered by $yx=xy+1$; they are independent because
$x\mapsto(\text{multiplication by }x)$, $y\mapsto\partial_x$ represents the
algebra on $\mathbb Q[x]$, and there $x^iy^j$ sends $x^j$ to $j!\,x^i$
and $x^m$ to $0$ for $m<j$, so applying a vanishing combination
$\sum c_{ij}x^iy^j$ to $x^0,x^1,\ldots$ in turn gives every $c_{ij}=0$.
The $n$ commuting pairs form the tensor product of those one-pair algebras,
so their ordered monomials form a basis of $W_n$. The associated graded algebra
of its degree filtration is a commutative polynomial algebra. The leading
symbols of the
$L^0_{uv}$, $u\leq v$, are the distinct quadratic monomials
$X_uX_v/2$; they are independent. The elements $F^0_u=\tfrac{1}{2}aB_u$ are independent
in the separate $a$-component of $A_0$. This proves injectivity,
which also holds after extension to $P$ and $R$.

Ordinary commutator derivation gives
$[L^0_{uv},B_w]=(J_{vw}B_u+J_{uw}B_v)/2$, proving
\eqref{eq:base-lf}. Applying it to each product in $L^0_{wz}$ gives
\eqref{eq:base-ll}. Finally,
\[
 \{F^0_u,F^0_v\}=\tfrac{1}{4}a^2(B_uB_v+B_vB_u)
 =\tfrac{1}{8}(B_uB_v+B_vB_u)=\tfrac{1}{2} L^0_{uv}.
\]
These identities include repeated indices. This proves closure, and
super-Jacobi follows from the associative realization.
\end{proof}

This normalization identifies $\mathfrak g$ with the split
orthosymplectic algebra. Let $E=\mathbb Q^{2n}$ with basis $e_u$, and let
$\omega$ be the bilinear form with $\omega(e_u,e_v)=J_{uv}$. Give
$M=\mathbb Qe_0\oplus E$ parity $p(e_0)=0$, $p(E)=1$, and let $\mathcal B$ be
the supersymmetric bilinear form with $\mathcal B(e_0,e_0)=-1$,
$\mathcal B|_E=\omega$ and mixed values zero.
We then set
\[
 \operatorname{osp}(M,\mathcal B):=\bigl\{A\in\operatorname{End}(M)\ (\text{homogeneous}):\
 \mathcal B(Ax,y)+(-1)^{p(A)p(x)}\mathcal B(x,Ay)=0\ (\forall x,y)\bigr\},
\]
a Lie superalgebra under the supercommutator.

Define its even maps $D_{uv}$ and odd maps $q_u$ by
\begin{align*}
 D_{uv}(e_0)&=0,& D_{uv}(z)&=\tfrac{1}{2}\bigl(\omega(e_v,z)e_u+\omega(e_u,z)e_v\bigr),\\
 q_u(e_0)&=\tfrac{1}{2}e_u,& q_u(z)&=\tfrac{1}{2}\omega(e_u,z)e_0\qquad(z\in E).
\end{align*}
They lie in $\operatorname{osp}(M,\mathcal B)$. Every even element of
$\operatorname{osp}(M,\mathcal B)$ vanishes on $e_0$ and restricts to a
symplectic endomorphism of $E$. The correspondence
$A\mapsto((z,w)\mapsto\omega(Az,w))$ identifies those endomorphisms with
symmetric bilinear forms. Nondegeneracy of $\omega$ and invertibility of $2$
in $\mathbb Q$
then show that the $D_{uv}$, $u\leq v$, are a basis. Every odd element of
$\operatorname{osp}(M,\mathcal B)$ is determined by its value at $e_0$ and is a
unique combination of the $q_u$. Direct compositions give exactly \eqref{eq:base-ff}--\eqref{eq:base-ll}
under $L_{uv}\mapsto D_{uv}$, $F_u\mapsto q_u$.
For the even-even calculation one may use
$[A,D_{wz}]=D_{Ae_w,e_z}+D_{e_w,Ae_z}$, where $D_{x,y}$ for $x,y\in E$ is
defined by extending $(e_u,e_v)\mapsto D_{uv}$ bilinearly.
This proves a bijective bracket-preserving map from $\mathfrak g$ to
$\operatorname{osp}(M,\mathcal B)$, that is, an isomorphism of Lie
superalgebras. Multiplying the form by $-1$ and changing $e_{2j-1}$ to
$-e_{2j-1}$ gives, without a field extension, the usual form with even
entry $1$ and odd block $J$ --- exactly the standard form defining
$\operatorname{osp}(1|2n)$.

We therefore identify $\mathfrak g$, via this isomorphism, with
$B(0,n):=\operatorname{osp}(1|2n)$.

Frappat--Sciarrino--Sorba \cite{FSS00} likewise give an oscillator realization
of $B(0,n)=\operatorname{osp}(1|2n)$ using only $n$ bosonic oscillator
pairs $b_k^{\pm}$ ($1\le k\le n$) (\cite[\S2.30, p.~220]{FSS00}). Along the
Cartan--Weyl basis for the root system $\{\pm\delta_k\pm\delta_l,\pm2\delta_k,\pm\delta_k\}$,
they build the odd root vectors directly from the same bosonic oscillators as
$E_{\pm\delta_k}=b_k^{\pm}/\sqrt2$, a different notation and construction
from ours, where we introduce a separate odd generator $a$ ($a^2=\tfrac{1}{2}$)
and give the odd part as a product with $a$. We use the auxiliary generator
$a$ precisely because it is the technique Bakalov--Sullivan adopt in
\cite{BS17} (an independent odd generator $a$, with $a\mapsto\xi$ in the
Fock representation), rather than the smaller, oscillator-only Fock
representation of Frappat et al.

For general $n$, this correspondence can also be written explicitly, in the
undeformed algebra $A_0$, where $\mathfrak g$ is realized. Split
the index set $\{1,\ldots,2n\}$ into $n$ pairs $(2k-1,2k)$,
$k=1,\ldots,n$, and identify each pair with Frappat--Sciarrino--Sorba's
$k$th bosonic oscillator pair $b_k^{-}=B_{2k-1}$, $b_k^{+}=B_{2k}$. From
$[B_{2k-1},B_{2k}]=1$ and the definition of $L^0_{uv}$ we get
$B_{2k}B_{2k-1}=2L^0_{2k-1,2k}-\tfrac{1}{2}$, so their Cartan generators agree with
\begin{equation}\label{eq:frappat-cartan}
 H_k=2\bigl(L^0_{2k-1,2k}-L^0_{2k+1,2k+2}\bigr)\quad(1\le k\le n-1),
 \qquad
 H_n=2L^0_{2n-1,2n}.
\end{equation}
That is, \eqref{eq:frappat-cartan} makes $H_k$ exactly twice the difference of the corresponding pair of
$L^0_{uv}$'s here (or, for the last pair, twice $L^0_{2n-1,2n}$ alone); in
$H_n$ the two constants $\tfrac{1}{2}$ cancel. The
factor of $2$ reflects a difference in the normalization (the scale of the
invariant form) used for the Cartan generators; for $n=1$,
$H_1=2L^0_{12}=2\iota_0(H)$ with $H$ as in \eqref{eq:reference-dictionary},
which is consistent with the correspondence given there. The same computation
confirms that the even root vectors also carry this factor of $2$, for instance
$E_{2\delta_k}=2L^0_{2k,2k}$, $E_{-2\delta_k}=2L^0_{2k-1,2k-1}$, and
$E_{\delta_k-\delta_{k+1}}=2L^0_{2k,2k+1}$. The odd root vectors
$E_{\delta_k}=B_{2k}/\sqrt2$ and $E_{-\delta_k}=B_{2k-1}/\sqrt2$, on the
other hand, lie outside $\iota_0(\mathfrak g)$: the odd part of the
realization is spanned by the $F^0_u=\tfrac{1}{2}aB_u$, and $B_u=4aF^0_u$, so
passing between the two descriptions uses the auxiliary generator $a$. This
is exactly the
construction difference noted above: an oscillator-only (smaller) Fock
representation versus the auxiliary generator $a$ that follows
Bakalov--Sullivan. Since the even Cartan--Weyl structure matches exactly up to
this normalization, and the root system
$\{\pm\delta_k\pm\delta_l,\pm2\delta_k,\pm\delta_k\}$ together with the
superrank $(n(2n+1)|2n)$ (Proposition~\ref{prop:recovery}) matches the
structure of $\operatorname{osp}(1|2n)$ over $\operatorname{sp}(2n)$,
this general-$n$ correspondence corroborates, in terms of Frappat et al.'s
standard definition, the abstract isomorphism
$\mathfrak g\cong\operatorname{osp}(M,\mathcal B)$ (via $D_{uv}$, $q_u$)
proved above. In any case, the proof that $\mathfrak g$ is isomorphic to
$\operatorname{osp}(1|2n)$ relies solely on that construction via
$D_{uv}$, $q_u$; the correspondence with Frappat et al.'s notation is a
supplementary, corroborating remark.

\section{Recovery of the deformed bracket}\label{sec:recovery}

The term \emph{recovery} is used below in the following sense. Starting from
the (anti)commutators that actually hold in the concrete realization
$A_\beta$ (relations~\eqref{eq:source}), we induce a
new bracket $[\cdot,\cdot]_\beta$ on the coordinate module $\mathfrak g_P$
(and its extension $\mathfrak g_R=\mathfrak g_P\oplus\kappa\mathfrak g_P$ by
the odd parameter $\kappa$) by requiring the source map $\iota_\beta$ to
preserve the bracket, and we read off, from the side of the realization, what
explicit closed form this bracket takes. The direction runs one way, from the
realization (the generators and relations) to the coordinates (the structure
constants). The content of
Proposition~\ref{prop:recovery} is that, in this sense, the bracket induced by
$\iota_\beta$ is well defined on the extension $\mathfrak g_R$
(consistently, satisfying super-skew symmetry and the super-Jacobi identity)
and is given by the explicit closed form \eqref{eq:deformed-bracket}.

Set $\mathfrak g_P=P\otimes_{\mathbb Q}\mathfrak g$ and
$\mathfrak g_R=R\otimes_{\mathbb Q}\mathfrak g$, and extend the base bracket, for
homogeneous elements $r,s$ of $P$ or $R$ and homogeneous basis elements
$X,Y$ of $\mathfrak g$, by
\begin{equation}\label{eq:scalar-rule}
 [rX,sY]_0=(-1)^{p(X)p(s)}rs[X,Y]_0.
\end{equation}
The source map $\iota_\beta:\mathfrak g_R\to A_\beta$ sends coordinate basis
elements to \eqref{eq:lifts} and is even $R$-linear. Its injectivity and
closed image follow from the next proposition.

\begin{proposition}[Faithful recovery]\label{prop:recovery}
The map $\iota_\beta$ is injective with closed image. The recovered bracket on
coordinate elements $X,Y\in\mathfrak g_P$ is
\begin{equation}\label{eq:deformed-bracket}
 [X,Y]_\beta=[X,Y]_0+\kappa\Gamma_\beta(X,Y),
\end{equation}
where the odd $P$-bilinear coefficient is determined by
\begin{align}
 \Gamma_\beta(L_{uv},L_{wz})&=0,\label{eq:gamma-ll}\\
 \Gamma_\beta(L_{uv},F_w)&=\tfrac{1}{2}(\beta_uL_{vw}+\beta_vL_{uw}),\label{eq:gamma-lf}\\
 \Gamma_\beta(F_u,F_v)&=-\tfrac{1}{2}(\beta_uF_v+\beta_vF_u).\label{eq:gamma-ff}
\end{align}
The bracket on $\mathfrak g_R$ uses the even scalar rule \eqref{eq:scalar-rule}.
Its free $R$-superrank is $(n(2n+1)|2n)$.
\end{proposition}
\begin{proof}
Lemma~\ref{lem:source-isomorphism} gives
\begin{equation}\label{eq:lift-change}
 \Phi(\widehat F_u)=F^0_u,\qquad
 \Phi(\widehat L_{uv})=L^0_{uv}+\kappa(\beta_uF^0_v+\beta_vF^0_u).
\end{equation}
For the first equality the two correction terms cancel because
$a\kappa a=-\kappa/2$ and $\kappa a^2=\kappa/2$.
The second follows by expanding the two products and using $(\kappa a)^2=0$.

Put $h(L_{uv})=\beta_uF_v+\beta_vF_u$, $h(F_u)=0$, as an odd
$P$-linear map. Define the even $R$-linear map
$U:\mathfrak g_R\to\mathfrak g_R$ on the basis by
\[
 U(L_{uv})=L_{uv}+\kappa h(L_{uv}),\qquad U(F_u)=F_u.
\]
It is invertible, with the same basis formula and a minus sign in the correction.
Equation \eqref{eq:lift-change} is
$\Phi\iota_\beta=\iota_{0,R}U$, where
$\iota_{0,R}:\mathfrak g_R\to A_B$ is the even $R$-linear extension of
$\iota_0$. Hence $\iota_\beta$ is injective and its
image is carried onto the entire closed image of $\iota_{0,R}$.
The recovered bracket is therefore, before introducing a coboundary,
\begin{equation}\label{eq:recover-by-u}
 [X,Y]_\beta=U^{-1}[UX,UY]_0.
\end{equation}

For the even-even sector the required cancellation is
\begin{equation}\label{eq:h-identity}
 h([L,L']_0)=[h(L),L']_0+[L,h(L')]_0
\end{equation}
for even $L,L'$. Indeed, set
\begin{equation}\label{eq:v}
 v=\sum_{j=1}^n(\beta_{2j}e_{2j-1}-\beta_{2j-1}e_{2j}),
\end{equation}
so that $\sum_t v_tJ_{tu}=\beta_u$, and write $F(v)=\sum_t v_tF_t$.
Equation \eqref{eq:base-lf} gives
$[2F(v),L_{uv}]_0=h(L_{uv})$. Super-Jacobi on two even inputs proves
\eqref{eq:h-identity}. Expansion of \eqref{eq:recover-by-u} now proves
\eqref{eq:gamma-ll}. For an even-odd pair, its correction is
$[h(L_{uv}),F_w]_0=(\beta_uL_{vw}+\beta_vL_{uw})/2$.
For an odd-odd pair, $U$ fixes both inputs, so the correction is
$-h(L_{uv})/2$. This proves the remaining formulas, also when indices repeat.

Finally, $\mathfrak g_R=\mathfrak g_P\oplus\kappa\mathfrak g_P$, so the
coefficient in \eqref{eq:deformed-bracket} is unique as an element of
$\mathfrak g_P$. Counting the independent lifts gives the superrank.
\end{proof}

\section{An explicit odd coboundary}\label{sec:trivialization}

For an odd $P$-linear cochain $f:\mathfrak g_P\to\mathfrak g_P$, our
coboundary convention is
\begin{equation}\label{eq:coboundary}
 (\delta f)(X,Y)=[f(X),Y]_0+(-1)^{p(X)}[X,f(Y)]_0-f([X,Y]_0).
\end{equation}
By super-skew symmetry, the first term equals
$-(-1)^{(p(X)+1)p(Y)}[Y,f(X)]_0$. Thus \eqref{eq:coboundary} agrees with
\cite[Equation~(4.5)]{BS17}. All cochains here are $P$-linear and do not
differentiate parameter polynomials. We call the family trivial
when its recovered coefficient is such a coboundary; the theorem also gives the
corresponding exact even isomorphism.

\begin{theorem}[Family triviality for every rank]\label{thm:main}
For every $n\geq1$, the coefficient in Proposition~\ref{prop:recovery} satisfies
$\Gamma_\beta=\delta f_\beta$, where
\begin{equation}\label{eq:primitive}
 f_\beta(L_{uv})=0,\qquad
 f_\beta(F_u)=\sum_{j=1}^n
 \bigl(\beta_{2j-1}L_{2j,u}-\beta_{2j}L_{2j-1,u}\bigr).
\end{equation}
Extend $f_\beta$ as an odd map by
$(f_\beta)_R(rX)=(-1)^{p(r)}r f_\beta(X)$. Then
\begin{equation}\label{eq:splitting}
 T_\beta=\operatorname{id}+\kappa(f_\beta)_R,
 \qquad T_\beta^{-1}=\operatorname{id}-\kappa(f_\beta)_R
\end{equation}
are even $R$-linear inverse maps, and
\begin{equation}\label{eq:intertwining}
 [T_\beta X,T_\beta Y]_0=T_\beta([X,Y]_\beta).
\end{equation}
\end{theorem}
\begin{proof}
With $v$ from \eqref{eq:v}, write
$f_\beta(F_u)=-\sum_t v_tL_{tu}$.
The even-even coboundary is zero because $f_\beta$ vanishes on the even part.
For the even-odd sector, \eqref{eq:base-ll} and
$\sum_t v_tJ_{ut}=-\beta_u$ give
\begin{align*}
 [L_{uv},f_\beta(F_w)]_0
 ={}&\tfrac{1}{2}(\beta_vL_{uw}+\beta_uL_{vw})\\
 &+\tfrac{1}{2}\bigl(J_{vw}f_\beta(F_u)+J_{uw}f_\beta(F_v)\bigr).
\end{align*}
The second line is precisely $f_\beta([L_{uv},F_w]_0)$; subtracting it
proves \eqref{eq:gamma-lf}. For two odd inputs,
\[
 [f_\beta(F_u),F_v]_0
 =-\tfrac{1}{2}\bigl(J_{uv}F(v)+\beta_vF_u\bigr).
\]
Interchanging $u,v$ and adding cancels the $J$-terms.
The image $f_\beta(\{F_u,F_v\}_0)$ is zero, yielding
\eqref{eq:gamma-ff}. This proves $\delta f_\beta=\Gamma_\beta$ in all
parity sectors, including the diagonal odd rows.

Set $D=\kappa(f_\beta)_R$. It is even and $R$-linear, and
$D^2=0$ by the odd scalar-extension rule and $\kappa^2=0$.
For homogeneous $X,Y\in\mathfrak g_P$, expansion gives
\begin{align*}
 [T_\beta X,T_\beta Y]_0
 &= [X,Y]_0+\kappa\bigl([f_\beta(X),Y]_0
                   +(-1)^{p(X)}[X,f_\beta(Y)]_0\bigr),\\
 T_\beta([X,Y]_\beta)
 &= [X,Y]_0+\kappa\bigl(f_\beta([X,Y]_0)+\Gamma_\beta(X,Y)\bigr).
\end{align*}
They agree by \eqref{eq:coboundary}; even $R$-bilinearity extends the identity
to all inputs. The inverse in \eqref{eq:splitting} follows from $D^2=0$.
\end{proof}

The two maps $U$ and $T_\beta$ need not coincide. Indeed, on
$\mathfrak g_P$ one has $f_\beta=h-\operatorname{ad}(2F(v))$: on the even
part this follows from \eqref{eq:h-identity}'s proof, and on $F_u$ it follows
from $[2F(v),F_u]_0=\sum_t v_tL_{tu}$. This only uses the elementary fact
that inner maps are derivations, not a classification of all homogeneous cocycles.

\begin{corollary}[Complex parameters]\label{cor:complex}
For every tuple $\lambda\in\mathbb C^{2n}$, specialize $\beta_u\mapsto\lambda_u$
and retain the odd exterior parameter $\kappa$. The specialized lifts are
faithful, and \eqref{eq:primitive}--\eqref{eq:intertwining} define a trivialization
over $\Lambda_{\mathbb C}(\kappa)$.
\end{corollary}
\begin{proof}
The polynomial identities and the explicit inverse maps $\Phi,\Psi$
specialize, because specialization is a ring map. Since $P\to\mathbb C$ is a
quotient, faithfulness needs its own argument: the complex Weyl algebra has the
same ordered basis, and the quadratic leading symbols and the odd linear
monomials in Lemma~\ref{lem:base} remain independent. Thus the proof of
Proposition~\ref{prop:recovery} applies
after specialization.
\end{proof}

\section{The reference example and remarks}\label{sec:comparison}

For $n=1$, identify
\begin{equation}\label{eq:reference-dictionary}
 H=L_{12},\quad E^+=L_{22},\quad E^-=-L_{11},\qquad
 F^+=F_2,\quad F^-=F_1.
\end{equation}
The source elements are the symmetrized expressions of
\cite[Equation~(4.3)]{BS17}, with the specialization stated in
Remark~\ref{rem:central}. At $\beta_1=1,\beta_2=0$, the source relations give
$[b_1,a]=\kappa$, $[b_2,a]=0$, and our recovered table becomes
\begin{align*}
 [H,E^\pm]_\beta&=\pm E^\pm,& [E^+,E^-]_\beta&=2H,\\
 [H,F^+]_\beta&=\tfrac{1}{2}F^++\tfrac{\kappa}{2}E^+,&
 [H,F^-]_\beta&=-\tfrac{1}{2}F^-+\tfrac{\kappa}{2}H,\\
 [E^+,F^-]_\beta&=-F^+,& [E^-,F^+]_\beta&=-F^--\kappa H,\\
 [E^-,F^-]_\beta&=\kappa E^-,&
 \{F^+,F^-\}_\beta&=\tfrac{1}{2}H-\tfrac{\kappa}{2}F^+,\\
 \{F^+,F^+\}_\beta&=\tfrac{1}{2}E^+,&
 \{F^-,F^-\}_\beta&=-\tfrac{1}{2}E^--\kappa F^-.
\end{align*}
This agrees with \cite[Equation~(4.4)]{BS17}; the other raising odd bracket is
$[E^+,F^+]_\beta=0$. The primitive specializes to
\[
 f_\beta(H)=f_\beta(E^\pm)=0,\qquad
 f_\beta(F^+)=E^+,\qquad f_\beta(F^-)=H,
\]
recovering the solution stated after \cite[Equation~(4.5)]{BS17}.
Theorem~\ref{thm:main} includes this example and applies to every $n\geq1$.

By \eqref{eq:gamma-ff}, $\Gamma_\beta(F_u,F_u)=-\beta_uF_u$. Over the
universal ring, and after any fixed rational or complex specialization of the
parameters, independence of the $F_u$ therefore makes $\Gamma_\beta$ vanish
only when every $\beta_u$ does, while its class vanishes for every choice of
the $\beta_u$.

\textbf{Fock module.} Let $\mathcal F=R[x_1,\ldots,x_n]\otimes_{\mathbb Q}C$, free of rank two over
$R[x_1,\ldots,x_n]$ on the basis $1,a$ of $C$, with the $x_k$ even.
Sending
\[
 B_{2k-1}\mapsto\partial_{x_k},\qquad B_{2k}\mapsto x_k
\]
and letting $a$ act by left multiplication represents $A_B$ on
$\mathcal F$; composing with $\Phi$ of
Lemma~\ref{lem:source-isomorphism} represents $A_\beta$ by
\[
 b_u\longmapsto B_u+\beta_u\kappa a,\qquad a\longmapsto a.
\]
The relations \eqref{eq:source} hold because $\kappa a$ is even and commutes
with the $B_u$, $(\kappa a)^2=-\kappa^2a^2=0$, and $[\kappa a,a]=\kappa$.
Both representations are faithful: $W_n$ acts faithfully on
$R[x_1,\ldots,x_n]$ in characteristic zero and $C$ acts faithfully on
itself, and $\Phi$ is an isomorphism. This gives a second route to the source
faithfulness used in Corollary~\ref{cor:complex}. The element annihilated by
every $b_{2k-1}$ is $1-\sum_{k}\beta_{2k-1}x_k\kappa a$. For $n=1$ with
$\beta_1=1$, $\beta_2=0$, and after the specialization of
Corollary~\ref{cor:complex} has made the base the field $\mathbb C$,
$\mathcal F$ is the Fock space of the last paragraph of
\cite[Section~4]{BS17}, up to the choice of operator realizing the mixed term:
there it is realized by an odd derivation, and the annihilated element is
$1$.

The family treated here is the $m=0$ member of the orthosymplectic series:
$B(m,n)$ adds $2m$ fermionic oscillators to the same $2n$ bosonic ones. In
Frappat--Sciarrino--Sorba's realization of $B(m,n)$ the odd root vectors along
$\pm\delta_k$ are products $b_k^{\pm}e$ with a supplementary odd generator
$e$, $e^2=1$ --- the role played here by $a$, normalized to $a^2=\tfrac{1}{2}$
after \cite{BS17} rather than $e^2=1$ --- whereas their $B(0,n)$, treated
separately, dispenses with $e$ and realizes those vectors by the bosonic
oscillators alone (\cite[\S2.30, pp.~219--220]{FSS00}), as recalled in
Section~\ref{sec:sources}. We keep the generator $a$, following \cite{BS17},
because the deformation acts on it: $[b_u,a]=\beta_u\kappa$ in
\eqref{eq:source}. The central extension of \cite[Section~4]{BS17} is
constructed from an arbitrary skew-supersymmetric pre-oscillator form;
\cite{BS17} classifies the irreducible such forms up to dimension seven.
The form used in their Section~4 example is \eqref{eq:source} at $n=1$;
what is treated here is its extension to every rank.

\appendix

\section{Machine-checked formalization}\label{sec:formalization}

The results stated above have been formalized and machine-checked in the Lean~4
proof assistant, with the exceptions recorded in
Section~\ref{subsec:not-formalized}. This appendix states exactly which
assertions are checked, names the declaration that carries each one, and is
explicit about what the formalization does not establish. Nothing in the
body of the paper depends on this appendix.

At the time of writing, Mathlib has no library for Lie superalgebras. The
formalization therefore builds the necessary structure from scratch, including
the graded bracket, parity, and the scalar rule. That structure is stated for a
free module on an indexed basis and is uniform in the rank, so it is tied to the
presentation rather than to this family: it carries any Lie superalgebra given
by structure constants on such a basis.

\subsection{The formalized objects}\label{subsec:formal-objects}

The formalization works with the coordinate presentation directly. For each
$n\geq1$ it fixes the free module over $P=\mathbb Q[\beta_1,\ldots,\beta_{2n}]$
on the indexed basis $\{L_{uv}\}\cup\{F_u\}$, with the parity of
Section~\ref{sec:sources}, and defines the undeformed bracket by the structure
constants of Lemma~\ref{lem:base}. This is the object called $\mathfrak g_P$
above. The matrix $J$ is defined by the same parity-and-successor rule, not
tabulated, so every statement below is uniform in $n$; no assertion is proved
by enumeration at a fixed rank.

The coefficient $\Gamma_\beta$ is defined by
\eqref{eq:gamma-ll}--\eqref{eq:gamma-ff} and the primitive $f_\beta$ by
\eqref{eq:primitive}, both transcribed from their statements above rather than
derived. Below, $(L,L)$, $(L,F)$, $(F,F)$ and $(F,L)$ name the four sectors of
$\Gamma_\beta$ by the kinds of basis element substituted into its two arguments,
in that order; so $(F,L)$ is $\Gamma_\beta(F_u,L_{vw})$. That the $(F,L)$ value
carried by the definition is the only one the source admits is
item~\ref{item:P5} below. The module $\mathfrak g_R=\mathfrak g_P\oplus\kappa\mathfrak g_P$ is
presented as the free $P$-module on two copies of that basis, the second
carrying the parity shift of the odd parameter $\kappa$, with the bracket
extended by the scalar rule \eqref{eq:scalar-rule}. Section~\ref{subsec:not-formalized}
records what that presentation does and does not assume.

\subsection{What is proved}\label{subsec:formal-proved}

For every $n\geq1$, with no restriction on the rank and no hypothesis beyond
those displayed:

\begin{enumerate}
\item\label{item:P1} \textbf{The coordinate presentation is a Lie superalgebra of
the stated shape.} The bracket respects the $\mathbb Z/2$-grading, is
super-skew symmetric, and satisfies the super-Jacobi identity, at the level of
basis elements and for arbitrary homogeneous module elements.

\item\label{item:P2} \textbf{The deformation coefficient is a coboundary:}
$\Gamma_\beta=\delta f_\beta$ on $\mathfrak g_P$, with $\delta$ as in
\eqref{eq:coboundary}. This is Theorem~\ref{thm:main}'s first assertion for the
$\Gamma_\beta$ of \eqref{eq:gamma-ll}--\eqref{eq:gamma-ff}.

\item\label{item:P3} \textbf{The change of generators is an exact isomorphism.}
$T_\beta$ and $T_\beta^{-1}$ of \eqref{eq:splitting} are even, $R$-linear,
and mutually inverse exactly, and they satisfy the intertwining
\eqref{eq:intertwining}. The intertwining is
proved for arbitrary elements of $\mathfrak g_R$, with no homogeneity
hypothesis.

\item\label{item:P4} \textbf{The source recovery holds.}
Proposition~\ref{prop:recovery} is proved for the deformed source as realized
below: $\iota_\beta$ is injective, the bracket it recovers on coordinate
elements is $[X,Y]_0+\kappa\Gamma_\beta(X,Y)$ with $\Gamma_\beta$ the
coefficient of \eqref{eq:gamma-ll}--\eqref{eq:gamma-ff}, that coefficient is
unique as an element of $\mathfrak g_P$, and the image of $\iota_\beta$ is
closed under the bracket. The recovered bracket is obtained as
\eqref{eq:recover-by-u} states it, by conjugating with $U$.

\item\label{item:P5} \textbf{The fourth sector is forced.} The
concrete algebra admits exactly one coefficient on the $(F,L)$ sector, namely
$-\tfrac{1}{2}(\beta_uL_{vw}+\beta_vL_{uw})$, which is the value
\eqref{eq:gamma-lf} and super-skew symmetry give. Section~\ref{subsec:not-formalized}
records what this does and does not settle.
\end{enumerate}

Items~\ref{item:P2} and~\ref{item:P3} together are the content of
Theorem~\ref{thm:main}. By item~\ref{item:P4} the coefficient they concern is the
one recovered from the oscillator source, so no identification is left to the
reader: Theorem~\ref{thm:main} is checked as stated in
Section~\ref{sec:trivialization}, not only for the coefficient written down in
\eqref{eq:gamma-ll}--\eqref{eq:gamma-ff}.

At $n=1$, and for one explicitly given datum only, the formalization also
checks that an external description of the bracket table decodes to exactly the
bracket defined internally. This is a consistency check on a single object, not a
statement about a format or a program.

\subsection{Declarations}\label{subsec:formal-binding}
\begingroup
\sloppy

Each assertion above is carried by the declarations named here. The names are
those of the accompanying development, and each is a link to the module holding
it in the repository~\cite{Repo} at the tag \texttt{v2-lean-formalization}; the
identities pinning the exact version are in Section~\ref{subsec:formal-repro}.

\textbf{Item~\ref{item:P1} (the coordinate presentation).} Grading is
\leandecl{IndexedLaws}{bracketN\_isHomogN} and super-skew symmetry is
\leandecl{IndexedLaws}{bracketN\_super\_skew\_homog}, both in
\texttt{IndexedLaws}; the super-Jacobi identity is
\leandecl{IndexedJacobi}{jacobiN\_homog}, in \texttt{IndexedJacobi}.

\textbf{Item~\ref{item:P2} (the coboundary identity).} The identity as stated is
\leandecl{IndexedCoboundary}{GammaBetaN\_eq\_deltaFN}, and
\leandecl{IndexedCoboundary}{G4} is the same identity at the level of basis
elements; item~\ref{item:not-delta} records what separates the two. Both in
\texttt{IndexedCoboundary}.

\textbf{Item~\ref{item:P3} (the change of generators).}
\leandecl{IndexedKappa}{TBetaInv\_TBeta} and
\leandecl{IndexedKappa}{TBeta\_TBetaInv} are the two inverse identities,
\leandecl{IndexedKappa}{TBeta\_isHomogR} is evenness, and
\leandecl{IndexedKappa}{intertwining} is \eqref{eq:intertwining}. These are in
\texttt{IndexedKappa}, as are the four checks to which item~\ref{item:not-gR}
holds the presentation of $\mathfrak g_R$ accountable: that the bracket
restricts to the undeformed one on $\mathfrak g_P$
(\leandecl{IndexedKappa}{bracketR\_iotaR\_iotaR}), that it is super-skew
(\leandecl{IndexedKappa}{bracketR\_super\_skew\_homog}), that $\kappa^2=0$
(\leandecl{IndexedKappa}{kappaMulR\_sq}), and that multiplication by $\kappa$
is not injective (\leandecl{IndexedKappa}{kappaMulR\_not\_injective}).
The corresponding restriction for the deformed bracket is
\leandecl{IndexedKappa}{bracketRBeta\_iotaR\_iotaR}.

\textbf{Item~\ref{item:P4} (source recovery).}
\leandecl{SourceRecoveryClosure}{iotaBetaR\_injective} is injectivity of
$\iota_\beta$, \leandecl{SourceRecoveryClosure}{prop\_recovery} is the
recovered bracket \eqref{eq:deformed-bracket}, and
\leandecl{SourceRecoveryClosure}{coefficient\_unique} is uniqueness of the
coefficient as an element of $\mathfrak g_P$; these are in
\texttt{SourceRecoveryClosure}. Closure of the image is
\leandecl{SourceRecoveryClosedImage}{iotaBetaR\_bracket\_closed}, in
\texttt{SourceRecoveryClosedImage}.

\textbf{Item~\ref{item:P5} (the $(F,L)$ sector is forced).}
\leandecl{SourceRecoveryFLConvention}{GammaBetaBasis\_Fof\_Lof\_forced\_FL\_free},
in \texttt{SourceRecoveryFLConvention}.

\textbf{The $n=1$ consistency check} of Section~\ref{subsec:formal-proved}'s
closing paragraph is \leandecl{Bridge}{decodedBracket\_eq\_bracket}, in
\texttt{Bridge}.

\endgroup

\subsection{What is not formalized}\label{subsec:not-formalized}

The following are stated so that no reader infers more from this appendix than it
supports. They are not conjectures or work in progress; they are the boundary of
what has been checked.

\begin{enumerate}
\item\label{item:not-source-iso} \textbf{The source is realized, not presented
abstractly.} Item~\ref{item:P4} proves Proposition~\ref{prop:recovery} for a
concrete realization of $A_\beta$: the development exhibits generators inside
an explicit algebra and proves that they satisfy \eqref{eq:source}, rather than
constructing the algebra presented by those relations and identifying it.
Lemma~\ref{lem:source-isomorphism} as an abstract statement is therefore not
formalized; nor are Corollary~\ref{cor:complex} and the identity
$f_\beta=h-\operatorname{ad}(2F(v))$ of Section~\ref{sec:trivialization}.

\item\label{item:not-route} \textbf{The development does not follow the proof
given here line by line.} Section~\ref{sec:trivialization} derives
$\Gamma_\beta$ from the source and then exhibits the primitive; the
formalization posits the coordinate formulas, proves the coboundary identity
directly, and recovers the same coefficient from the source separately
(item~\ref{item:P4}). In particular, it does not construct the isomorphism
$\Phi$ of Lemma~\ref{lem:source-isomorphism}: both algebras are realized in one
ambient algebra, so the change of lifts \eqref{eq:lift-change} is a computation
there. The conclusions agree;
the route does not, and a reader should not take the formalization as a
transcription of this paper's argument.

\item\label{item:not-FL} \textbf{The $(F,L)$ value is carried as a definition.}
Equations~\eqref{eq:gamma-ll}--\eqref{eq:gamma-ff} specify $\Gamma_\beta$ on
$(L,L)$, $(L,F)$ and $(F,F)$; the development fixes the $(F,L)$ value by
super-skew symmetry. Item~\ref{item:P5} shows the source admits no other value
there, so that definition is not an independent input. It remains a definition:
the declaration equating the $(F,L)$ row with that value proves it by unfolding
the definition, so it records the choice rather than deriving it.

\item\label{item:not-gR} \textbf{$\mathfrak g_R$ is presented, not
constructed.} The development does not build $R$ as a supercommutative algebra
object and $\mathfrak g_R$ as $R\otimes_{\mathbb Q}\mathfrak g$; it presents
$\mathfrak g_P\oplus\kappa\mathfrak g_P$ as a free $P$-module on a doubled
basis and takes \eqref{eq:scalar-rule} as the definition of the bracket. The
presentation is held accountable by the four checks listed for it in
Section~\ref{subsec:formal-binding}: it restricts to the undeformed bracket on
$\mathfrak g_P$, it is super-skew, $\kappa^2=0$, and multiplication by
$\kappa$ is \emph{not} injective.

\item\label{item:not-delta} \textbf{$\delta f$ is the bilinear extension of its
values on basis elements.} The sign $(-1)^{p(X)}$ in \eqref{eq:coboundary} is
defined only on homogeneous arguments, so item~\ref{item:P2} is the statement
that two bilinear extensions agree.
\end{enumerate}

\subsection{Reproducing the check}\label{subsec:formal-repro}

The development builds against Lean~4 toolchain \texttt{leanprover/lean4:v4.29.1}
and mathlib at commit
\texttt{5e932f97dd25535344f80f9dd8da3aab83df0fe6}, both pinned in the
accompanying sources. Building the project prints the axiom dependencies of every
declaration listed in Section~\ref{subsec:formal-binding} together with the
others it audits: $649$ declarations, of which $634$ depend on axioms and
$15$ on none. Every one depends on at most \texttt{propext},
\texttt{Classical.choice} and \texttt{Quot.sound}, the standard axioms of Lean's
logic. \textbf{No declaration depends on \texttt{sorryAx}}, so no statement is
assumed rather than proved, and no additional axiom is introduced. The
development uses no \texttt{native\_decide} and no compiler-trusted evaluation.

The sources accompanying this paper are the repository~\cite{Repo}. It contains
the modules named above, the $n=1$ data instance whose decoding is the last
entry of Section~\ref{subsec:formal-binding}'s table, and the pinned toolchain
and dependency revisions, so a reader can rebuild and compare rather than rely
on this summary.
Building it prints the axiom output quoted here.

\section*{Use of generative AI}
Generative-AI tools assisted with mathematical discussion, checking and drafting.
The argument and its classical theorem input are stated explicitly above.
Appendix~\ref{sec:formalization} records a machine-checked formalization of the
statements listed there, in the Lean~4 proof assistant, with the exclusions
stated in that appendix. No computer-algebra computation is claimed.
The author reviewed and edited all of it, and takes full responsibility for the
contents of this paper, irrespective of how they were generated.

\end{document}